\documentclass{birkjour}  %

 \newtheorem{thm}{Theorem}[section]
 \newtheorem{cor}[thm]{Corollary}

 \theoremstyle{definition}
 \newtheorem{defn}[thm]{Definition}
 \theoremstyle{remark}

 \numberwithin{equation}{section}

 \usepackage{graphicx}
\usepackage{amsmath}
\usepackage{amssymb}
\usepackage{enumerate}
\usepackage{rawfonts}
\usepackage{setspace}

\usepackage{arydshln}
\usepackage{lscape}

\usepackage{float}
\usepackage{color}
\usepackage{multicol}

\usepackage[all]{xy}
\xyoption{poly} \xyoption{graph} \xyoption{arc} \xyoption{web}
\xyoption{2cell}

\begin{document}

%-------------------------------------------------------------------------
% editorial commands: to be inserted by the editorial office
%
%\firstpage{1} \volume{228} \Copyrightyear{2004} \DOI{003-0001}
%
%
%\seriesextra{Just an add-on}
%\seriesextraline{This is the Concrete Title of this Book\br H.E. R and S.T.C. W, Eds.}
%
% for journals:
%
%\firstpage{1}
%\issuenumber{1}
%\Volumeandyear{1 (2004)}
%\Copyrightyear{2004}
%\DOI{003-xxxx-y}
%\Signet
%\commby{inhouse}
%\submitted{March 14, 2003}
%\received{March 16, 2000}
%\revised{June 1, 2000}
%\accepted{July 22, 2000}
%
%
%
%---------------------------------------------------------------------------
%Insert here the title, affiliations and abstract:
%

\title{$k$-connected degree sequences}

%----------Author 1
\author{Jonathan McLaughlin}

\address{%
Department of Mathematics, \\ 
St. Patrick's College,\\
 Dublin City University,\\ 
 Dublin 9, \\ 
 Ireland }

\email{jonny$\_\;$mclaughlin@hotmail.com}

%\thanks{}
%----------Author 2

%----------classification, keywords, date


\date{\today}
%----------additions
%\dedicatory{}
%%% ----------------------------------------------------------------------

\begin{abstract} Necessary and sufficient conditions for a sequence of positive integers to be the degree sequence of a $k$-connected simple graph are detailed. Conditions are also given under which such a sequence is necessarily $k$-connected. 
\end{abstract}

%%% ----------------------------------------------------------------------
\maketitle
%%% ----------------------------------------------------------------------

\section{Introduction}
\parindent=0cm

Necessary and sufficient conditions for a sequence of non-negative integers to be connected i.e. the degree sequence of some finite simple connected graph, are implicit in Hakimi \cite{Hk} and have been stated explicitly by the author in \cite{Me15} and expounded on in  \cite{Me1c}. This note builds upon these conditions of Hakimi and details necessary and sufficient conditions for a sequence of non-negative integers to be $k$-connected. The note concludes with necessary and sufficient conditions for a sequence of non-negative integers to be necessarily $k$-connected i.e. the sequence can only be realised as a $k$-connected graph.

\section{Preliminaries }

Let $G=(V_{G},E_{G})$ be a graph where $V_{G}$ denotes the vertex set of $G$ and $E_{G}\subseteq [V_{G}]^{2}$ denotes the edge set of $G$ (given that $[V_G]^2$ is the set of all $2$-element subsets of $V_G$).  An edge $\{a,b\}$ is denoted $ab$. A graph is finite when $|V_{G}|<\infty$ and $|E_{G}|<\infty$, where $|X|$ denotes the cardinality of the set $X$. The union of graphs $G$ and $H$ is the graph $G\cup H=(V_{G}\cup V_{H}, E_{G}\cup E_{H})$ and $G\cup ab$ is understood to be the graph $(V_G, E_G)\cup (\{a,b\},\{ab\})$. A graph is simple if it contains no loops (i.e. $aa\not\in E_{G}$) or parallel/multiple edges (i.e. $\{ab,ab\}\not\subseteq E_{G}$). The {\it degree} of a vertex $v$ in a graph $G$, denoted $deg(v)$, is the number of edges in $G$ which contain $v$. A graph where all vertices have degree $k$ is called a {\it $k$-regular} graph. A {\it path} is a graph with $n$ vertices in which two vertices, known as the {\it endpoints}, have degree $1$ and $n-2$ vertices have degree $2$. A graph is {\it connected} if there exists at least one path between every pair of vertices in the graph. Paths $P_1$ and $P_2$, both with endpoints $a$ and $b$, are {\it internally disjoint} if $P_1\cap P_2=(\{a,b\},\{\})$. A graph $G$ is $k${\it -connected} when there exists at least $k$ internally disjoint paths in $G$ between any two vertices in $G$. This characterisation of a graph being $k$-connected is synonymous with Menger's Theorem. $K_{n}$ denotes the {\it complete graph} on $n$ vertices. The {\it complement} of a simple graph $G$ is the simple graph $\overline{G}$ with vertex set $V_G$ and edge set the pairs of vertices in $V_G$ which are not contained in $E_G$. All basic graph theory definitions can be found in standard texts such as \cite{BM}, \cite{D} or \cite{GG}.

\section{Degree sequences }\label{s5}

A finite sequence $s=\{s_1,...,s_n\}$ of non-negative integers is called {\it graphic} if there exists a finite simple graph with vertex set  $\{v_1,..., v_n\}$ such that $v_i$ has degree $s_i$ for all $i=1,...,n$. Given a graph $G$ then the degree sequence $d(G)$ is the monotonic non-increasing sequence of degrees of the vertices in $V_G$. This means that every graphical sequence $s$ is equal to the degree sequence $d(G)$ of some graph $G$ (subject to possible rearrangement of the terms in $s$).

\begin{defn} A finite sequence $s=\{s_1,...,s_n\}$ of positive integers is called {\it $k$-connected} if there exists a finite simple $k$-connected graph with vertex set  $\{v_1,..., v_n\}$ such that $deg(v_i)= s_i$ for all $i=1,...,n$.
\end{defn}
A finite sequence $s=\{s_1,...,s_n\}$ of positive integers is called {\it necessarily $k$-connected} if $s$ can only be realisable as a simple $k$-connected graph. \\

Given a sequence of positive integers $s=\{s_1,...,s_n\}$ then define the {\it associated pair of $s$}, denoted $(\varphi(s),\epsilon(s))$, to be the pair $(n, \frac{1}{2}\sum\limits_{i=1}^{n}s_i)$. Where no ambiguity can arise,  $(\varphi(s),\epsilon(s))$ is simply denoted $(\varphi,\epsilon)$.

\section{Results}

\begin{thm}\label{Main} Given a sequence $s=\{s_1,...,s_n\}$ of positive integers, with the associated pair $(\varphi, \epsilon)$, such that $s_i\geq s_{i+1}$ for $i=1,...,n-1$ then $s$ is $k$-connected if and only if 
\begin{itemize}
\item $\epsilon \in \mathbb{N}$,
\item $s_1\leq \varphi-1$ and $s_n\geq k$,
\item $\frac{k\varphi}{2} \leq \epsilon \leq {{\varphi \choose 2}}$. 
\end{itemize}
\end{thm}

\begin{proof} ($\Rightarrow$) Clearly $\epsilon \in \mathbb{N}$ is a necessary condition for any sequence $s$ to be realisable as half the sum of the degrees in any graph is the number of edges in that graph which must be a natural number. The necessity of the condition $s_1\leq \varphi-1$ follows from the definition of a simple graph and the need for $s_n\geq k$ is due to the fact that every vertex in a $k$-connected graph has degree at least $k$. \\

The necessity of the condition $ \epsilon \geq \frac{k\varphi}{2} $ follows from the observation that the minimum possible $\epsilon$ of any $k$-connected sequence is $\epsilon=\frac{k\varphi}{2}$ which is contained in the associated pair of $s=\underbrace{\{k,...,k\}}_{\varphi}$, the degree sequence of a $k$-regular graph. Note that $\epsilon$ must be even, by definition, and this occurs whenever $\varphi$ and $k$ are not {\it both} odd (as there cannot exist a $k$-regular graph, where $k$ is odd and the number of vertices is also odd). If both $\varphi$ and $k$ are odd, then the minimum possible $\epsilon$ of any $k$-connected sequence is $\epsilon=\frac{k\varphi+1}{2}$ which is contained in the associated pair of $s=\underbrace{\{k+1,k,...,k\}}_{\varphi}$. To show the necessity of the condition $\epsilon \leq {{\varphi \choose 2}}$, the parity of $\varphi$ is irrelevant. When maximising $\epsilon$ it follows that the maximum possible $\epsilon$ of any $k$-connected sequence is $\epsilon={\varphi \choose 2}$ which is contained in the associated pair of $s=\underbrace{\{\varphi-1,...,\varphi-1\}}_{\varphi}$, the degree sequence of the complete graph $K_{\varphi}$. \\

It remains to show that, for a fixed $\varphi$, all sequences with $\frac{k\varphi}{2} < \epsilon < {{\varphi \choose 2}}$ are realisable. Suppose that $\varphi$ and $k$ are not both odd and $k<\varphi -1$. To ensure that $\epsilon$ remains even, two terms in $\{k,...,k\}$ must be incremented by one to give $\{k+1,k+1,k,...,k\}$. Given that $\{k,...,k\}$ is realisable as a $k$-regular graph $G$ then $\{k+1,k+1,k,...,k\}$ is realisable as $G\cup ab$ where $ab\in \overline{G}$. This incrementing of two terms in a graphic sequence can be continued until the sequence with $\epsilon = {{\varphi \choose 2}}$ is reached and this process is summarised in {\sc Figure} \ref{table1}. 

 \begin{figure}[h]
{\renewcommand{\arraystretch}{1.5} 
\[ \begin{array}{l|c|c} 
 s=\{s_1,\dots, s_n\} &   \epsilon & |E_G| \\ \hline
 \{k,\dots,k\}  &   \frac{k\varphi}{2} &   \frac{kn}{2}   \\
 \{k+1,k+1,k,\dots,k\}  &   \frac{k\varphi +2}{2} &   \frac{kn}{2} +1   \\
  \hspace{0.5cm}\vdots  &    \vdots &    \vdots \\
  \{n-1,\dots,n-1\} &  {\varphi \choose 2} &  {n \choose 2}  \\
\end{array}  \] }

\caption{$k$-connected degree sequences when $\varphi$ and $k$ are not both odd.}
\label{table1}
\end{figure}

A similar argument exists when $\varphi$ and $k$ are both odd, with $k<\varphi -1$, and this argument is summarised in {\sc Figure} \ref{table2}. \\
 
 \begin{figure}[h]
{\renewcommand{\arraystretch}{1.5} 
\[ \scalebox{0.95}{$\begin{array}{l|c|c} 
 s=\{s_1,\dots, s_n\} &   \epsilon & |E_G| \\ \hline
 \{k+1,k,\dots,k\}  &   \frac{k\varphi+1}{2} &   \frac{kn+1}{2}   \\
 \{k+2,k+1,k,\dots,k\},\;  \{k+1,k+1,k+1,k,\dots,k\}  &   \frac{(k\varphi+1)+2}{2} &   \frac{kn+1}{2} +1   \\
  \hspace{0.5cm}\vdots  &    \vdots &    \vdots \\
  \{n-1,\dots,n-1\} &  {\varphi \choose 2} &  {n \choose 2}  \\
\end{array} $} \] }

\caption{$k$-connected degree sequences when $\varphi$ and $k$ are both odd.}
\label{table2}
\end{figure}

($\Leftarrow$) Suppose that $s=\{s_1,...,s_n\}$ is $k$-connected. This means that $s$ is the degree sequence of a $k$-connected graph $G$, hence $\sum\limits_{i=1}^{n}deg(v_i)=2|V_G|$ and so $\epsilon\in\mathbb{N}$. As $G$ is $k$-connected then $deg(v_i)\geq k$ for all $i=1,...,n$ hence if $G$ is a minimally $k$-connected graph on $n$ vertices then $d(G)=\{k,...,k\}$ with $|E_{G}|=\frac{kn}{2}$ if $n$ and $k$ are not both odd or $d(G)=\{k+1,k,...,k\}$ with $|E_{G}|=\frac{kn+1}{2}$ if $n$ and $k$ are both odd, hence $s_n\geq k$ and $\epsilon \geq \frac{k\varphi}{2}$. As $G$ is simple then $deg(v_i)\leq n-1$ for all $i=1,...,n$ and the maximal simple ($k$-connected) graph on $n$ vertices is the complete graph $K_n$ which has the degree sequence $\{n-1,...,n-1\}$ and $|E_{K_n}|={n \choose 2}$, hence $s_1\leq n-1$ and $\epsilon \leq {\varphi \choose 2}$. To show that $k$-connected graphs exist for each $|E_G|\in \mathbb{N}$ such that $\frac{kn}{2} < |E_G| < {{n \choose 2}}$ refer to the argument showing the existence of graphic sequences with $\frac{k\varphi}{2} < \epsilon < {{\varphi \choose 2}}$ detailed above, along with {\sc Figures} \ref{table1} and \ref{table2}. 
\end{proof}

\begin{thm}\label{Crry1} Given a sequence $s=\{s_1,...,s_n\}$ of positive integers, with the associated pair $(\varphi, \epsilon)$, such that $s_i\geq s_{i+1}$ for $i=1,...,n-1$ then $s$ is necessarily $k$-connected if and only if $s$ is $k$-connected and 
$\epsilon > {{\varphi -2 \choose 2}} + 2k -1.$
\end{thm}

\begin{proof}  ($\Rightarrow$) Clearly it is necessary for $s$ to be $k$-connected if it is to be necessarily $k$-connected. It is required to show that it is necessary for $ \epsilon > {{\varphi -2 \choose 2}} + 2k -1$. Consider a sequence $s=\{s_1,...,s_n\}$ such that $ \epsilon = {{\varphi -2 \choose 2}} + 2k -1.$ \\

Observe that one such sequence is $s'=\{\underbrace{n-1,...,n-1}_{k-1},\underbrace{n-3,...,n-3}_{n-k-1},\underbrace{k,k}_{2}\}$ where $\varphi(s')=k-1+n-k-1+2=n$ and $\epsilon(s')= \frac{(k-1)(n-1)+(n-k-1)(n-3)+2k }{2}$ $=$  $ \frac{(n-k-1+k-1)(n-3)+2(k+k-1) }{2} = \frac{(n-2)(n-3)+2(2k-1) }{2} = {n-2 \choose 2} +2k -1$ . \\

\begin{figure}[h]
\begin{center}
\scalebox{0.9}{$\begin{xy}\POS (-2,5) *\cir<2pt>{} ="a" *+!D{v_n},
 (-2,-5) *\cir<2pt>{} ="c" *+!U{v_{n-1}},
 (22,5) *\cir<2pt>{} ="b" *+!D{v_i},
(22,-5) *\cir<2pt>{} ="d" *+!U{v_j},
  (13,15)*+!{G_1},
   (-10,11)*+!{K_{k+1}},
    (35,11)*+!{K_{n-2}},
  (30,0)*+!{K_{n-k-1}},
   (10,0)*+!{K_{k-1}},
   (-7,0)*+!{K_{2}},
  
\POS "a" \ar@{-}  "c",
\POS "b" \ar@{-}  "d",

\POS(20,0),  {\ellipse(18,12)<>{}},
\POS(3,0),  {\ellipse(15,12)<>{}},

\POS\POS (68,5) *\cir<2pt>{} ="a" *+!D{v_n},
 (68,-5) *\cir<2pt>{} ="c" *+!U{v_{n-1}},
 (92,5) *\cir<2pt>{} ="b" *+!D{v_i},
(92,-5) *\cir<2pt>{} ="d" *+!U{v_j},
  (83,15)*+!{G_2},
   %(-10,11)*+!{K_{k+1}},
    %(35,11)*+!{K_{n-2}},
  %(30,0)*+!{K_{n-2}},
   %(10,0)*+!{K_{k-1}},
   %(-7,0)*+!{K_{2}},
  
\POS "a" \ar@{-}  "b",
\POS "c" \ar@{-}  "d",

\POS(90,0),  {\ellipse(18,12)<>{}},
\POS(73,0),  {\ellipse(15,12)<>{}},

 \end{xy}$}

\caption{ For $k\geq 1$ and $s'=\{n-1,...,n-1,n-3,...,n-3,k,k \}=d(G_1)=d(G_2)$.}
\label{graphs}
\end{center}
\end{figure}
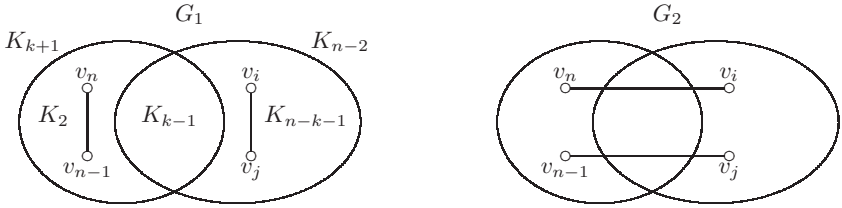

 Observe that $s'=d(G_1)$, see {\sc Figure} \ref{graphs}, where $G_1 = H_1\cup H_2$ with $H_1\simeq K_{k+1}$, $H_2\simeq K_{n-2}$ and $H_1\cap H_2 \simeq K_{k-1}$. Note that as $G_1\setminus (H_1\cap H_2)$ is disconnected and that $H_1\cap H_2 \simeq K_{k-1}$ (i.e. $|V_{H_1\cap H_2}|=k-1$) then $G_1$ is $(k-1)$-connected. \\ 

However, $s'=\{\underbrace{n-1,...,n-1}_{k-1},\underbrace{n-3,...,n-3}_{n-k-1},\underbrace{k,k}_{2}\}$ is, in fact, $k$-connected as $s'$ is also the degree sequence of $G_2$, see {\sc Figure} \ref{graphs} (noting that $v_iv_j\in E_{G_1}$ but $v_iv_j\not\in E_{G_2}$). Therefore, it is required that $ \epsilon > {n-2 \choose 2} +2k -1$ if $s$ is to be necessarily $k$-connected.  \\

Note that {\sc Figures} \ref{graphs1}, \ref{graphs2} and \ref{graphs3} are, respectively, the $k=1,2$ and $3$ versions of {\sc Figure} \ref{graphs}.

\begin{figure}[h]
\begin{center}
\scalebox{0.9}{$\begin{xy}\POS (-2,5) *\cir<2pt>{} ="a" *+!D{v_n},
(10.5,5) *\cir<2pt>{} ="b" *+!DR{v_i},
 (-2,-5) *\cir<2pt>{} ="c" *+!U{v_{n-1}},
(10.5,-5) *\cir<2pt>{} ="d" *+!UR{v_j},
  (-10,12)*+!{G_1},
  (19,0)*+!{K_{n-2}},
   (-7,0)*+!{K_{2}},
  
\POS "a" \ar@{-}  "c",
\POS "b" \ar@{-}  "d",

\POS(20,0),  {\ellipse(12,8)<>{}},

\POS (70,5) *\cir<2pt>{} ="a" *+!DR{v_n},
(80.5,5) *\cir<2pt>{} ="b" *+!DR{v_i},
 (70,-5) *\cir<2pt>{} ="c" *+!UR{v_{n-1}},
(80.5,-5) *\cir<2pt>{} ="d" *+!UR{v_j},
  (60,12)*+!{G_2},
  
\POS "a" \ar@{-}  "b",
\POS "c" \ar@{-}  "d",

\POS(90,0),  {\ellipse(12,8)<>{}},

 \end{xy}$}

\caption{ $k=1$ and $s'=\{n-3,...,n-3,1,1\}=d(G_1)=d(G_2)$. }
\label{graphs1}
\end{center}
\end{figure}
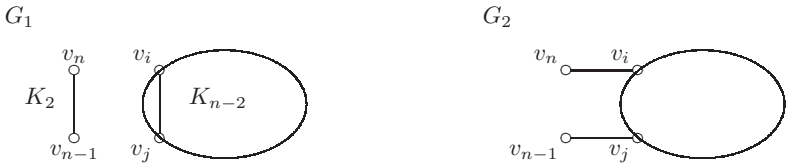

%%%

\begin{figure}[h]
\begin{center}
\scalebox{0.9}{$\begin{xy}\POS (-3,5) *\cir<2pt>{} ="a" *+!D{v_n},
(29.5,5) *\cir<2pt>{} ="b" *+!DL{v_i},
 (8,0) *\cir<2pt>{} ="e" *+!L{v_{1}},
 (-3,-5) *\cir<2pt>{} ="c" *+!U{v_{n-1}},
(29.5,-5) *\cir<2pt>{} ="d" *+!UL{v_j},
  (-10,12)*+!{G_1},
  (20,-1)*+!{K_{n-2}},
   (-7,-1)*+!{K_{3}},
  
\POS "a" \ar@{-}  "c",
\POS "b" \ar@{-}  "d",
\POS "a" \ar@{-}  "e",
\POS "c" \ar@{-}  "e",

\POS(20,0),  {\ellipse(12,8)<>{}},

\POS (67,5) *\cir<2pt>{} ="a" *+!DR{v_n},
(99.5,5) *\cir<2pt>{} ="b" *+!DL{v_i},
 (78,0) *\cir<2pt>{} ="e" *+!L{v_{1}},
 (67,-5) *\cir<2pt>{} ="c" *+!UR{v_{n-1}},
(99.5,-5) *\cir<2pt>{} ="d" *+!UL{v_j},
  (60,12)*+!{G_2},
  %(89,-1)*+!{K_{n-2}},
   %(71,-1)*+!{K_{3}},
  
\POS "d" \ar@{-}  "c",
\POS "b" \ar@{-}  "a",
\POS "a" \ar@{-}  "e",
\POS "c" \ar@{-}  "e",

\POS(90,0),  {\ellipse(12,8)<>{}},

 \end{xy}$}

\caption{ $k=2$ and $s'=\{n-1,n-3,...,n-3,2,2\}=d(G_1)=d(G_2)$. }
\label{graphs2}
\end{center}
\end{figure}
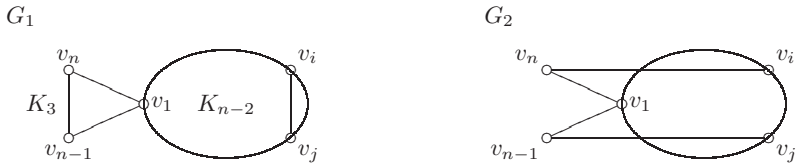

%%%

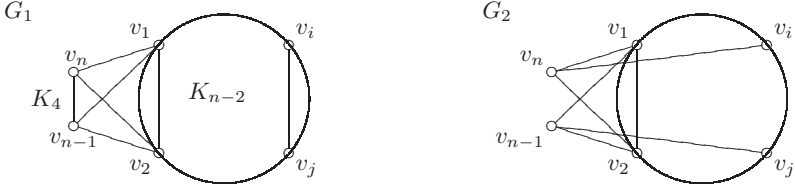
\begin{figure}[h]
\begin{center}
\scalebox{0.9}{$\begin{xy}\POS (-2,4) *\cir<2pt>{} ="a" *+!D{\;v_n},
 (-2,-4) *\cir<2pt>{} ="b" *+!U{v_{n-1}},
(10.5,8) *\cir<2pt>{} ="c" *+!DR{v_1},
(10.5,-8) *\cir<2pt>{} ="d" *+!UR{v_2},
 (29.5,8) *\cir<2pt>{} ="e" *+!DL{v_i},
(29.5,-8) *\cir<2pt>{} ="f" *+!UL{v_j},
  (-10,12)*+!{G_1},
  (19,0)*+!{K_{n-2}},
   (-6,-1)*+!{K_{4}},
  
\POS "a" \ar@{-}  "b",
\POS "a" \ar@{-}  "c",
\POS "a" \ar@{-}  "d",
\POS "b" \ar@{-}  "c",
\POS "b" \ar@{-}  "d",
\POS "c" \ar@{-}  "d",
\POS "e" \ar@{-}  "f",

\POS(20,0),  {\ellipse(12.5,12.5)<>{}},

\POS (68,4) *\cir<2pt>{} ="a" *+!DR{\;v_n},
 (68,-4) *\cir<2pt>{} ="b" *+!UR{v_{n-1}},
(80.5,8) *\cir<2pt>{} ="c" *+!DR{v_1},
(80.5,-8) *\cir<2pt>{} ="d" *+!UR{v_2},
 (99.5,8) *\cir<2pt>{} ="e" *+!DL{v_i},
(99.5,-8) *\cir<2pt>{} ="f" *+!UL{v_j},
  (60,12)*+!{G_2},
 % (89,0)*+!{K_{n-2}},
 %  (-10,0)*+!{K_{2}},
  
%\POS "a" \ar@{-}  "b",
\POS "a" \ar@{-}  "e",
\POS "b" \ar@{-}  "f",
\POS "a" \ar@{-}  "c",
\POS "a" \ar@{-}  "d",
\POS "b" \ar@{-}  "c",
\POS "b" \ar@{-}  "d",
\POS "c" \ar@{-}  "d",
%\POS "e" \ar@{-}  "f",

\POS(90,0),  {\ellipse(12.5,12.5)<>{}},

 \end{xy}$}

\caption{ $k=3$ and $s'=\{n-1,n-1,n-3,...,n-3,3,3\}=d(G_1)=d(G_2)$. }
\label{graphs3}
\end{center}
\end{figure}

($\Leftarrow$) It is now required to show that if $s$ is $k$-connected and $ \epsilon > {n-2 \choose 2} +2k -1$ then $s$ is necessarily $k$-connected. To show this it is required to show that the maximum number of edges in a graph with $n$ vertices which is not $k$-connected is ${n-2 \choose 2} +2k -1$. The graph $G_1$ in {\sc Figure} \ref{graphs} shows that such a graph exists, so it remains to show that a graph with $ \epsilon = {n-2 \choose 2} +2k -1$ is maximally $(k-1)$-connected i.e. adding one edge will always result in a $k$-connected graph. \\

Observe that any maximally $(k-1)$-connected graph $G$ on $n$ vertices will necessarily contain a cut set $C$ containing $k-1$ vertices. This means that $G\setminus C$ is disconnected. To maximise the number of edges in $G$ it is clear that $G\setminus C$ contains two connected components i.e. $G=H_1 \cup H_2$ where $V_{H_1}\cap V_{H_2}=C$ with $H_1\simeq K_{a+|C|}$, $H_2\simeq K_{b+|C|}$ and $H_1\cap H_2\simeq K_{|C|}$ (noting that $a+b=n-|C|$). So, the task of maximising $|E_{H_1\setminus C}|+|E_{H_2\setminus C}|$ is equivalent to minimising the number of edges in a complete bipartite graph $K_{a,b}$ as $K_{n}\setminus (E_{H_1\cup H_2}) \simeq K_{a,b}$. \\

Let $a+b=n-|C|$, with $a\leq b$, then $|E_{K_{a,b}}|=ab$ where $a,b\in\{1,...,n-|C|-1\}$. Note that $a>0$ as $G\setminus C$ is disconnected i.e. $K_a\neq K_0=(\varnothing, \varnothing)$. It is straightforward to show that $ab$ attains its maximum at $a=b=\frac{n-|C|}{2}$, when $n$ is even, and at $a=\lfloor\frac{n-|C|}{2}\rfloor, b=\lceil\frac{n-|C|}{2}\rceil$ when $n$ is odd. It follows that $ab$ is minimised when $a=1$ and $b=n-|C|-1$. However, observe that $a>1$ as $a=1$ implies that $H_1\simeq K_{|C|+1}$ which means that $d(G)$ contains a term equal to $|C|=k-1$, but this contradicts the $s_n\geq k$ condition. Hence $|E_{K_{a,b}}|$, with $a+b=n-|C|$, is minimised when $a=2$ and $b=n-|C|-2$ and so any maximally $(k-1)$-connected graph on $n$ vertices is isomorphic to $H_1\cup H_2$ where $H_1\simeq K_{k+1}, H_2\simeq K_{n-k-1}$ and $H_1\cap H_2\simeq K_{k-1}$, see $G_1$ in {\sc Figure} \ref{graphs}. Notice that the union of $G_1$ and any edge in $\overline{G_1}$ results in a $k$-connected graph. 
\end{proof}

\begin{cor} All simple graphs with $n$ vertices and at least $\frac{n^2-5n +6+4k}{2}$ edges are $k$-connected.
\end{cor}
\begin{proof}  As shown in Theorem \ref{Crry1}, a maximally $(k-1)$-connected graph with $n$ vertices is isomorphic to the union of $H_1\simeq K_{k+1}$ and $H_2\simeq K_{n-2}$ where $H_1\cap H_2 \simeq K_{k-1}$ and such a graph has $|E_{K_{k+1}}|+|E_{K_{n-2}}|-|E_{K_{k-1}}| = {n -2 \choose 2} +{k+1 \choose 2} -{k-1 \choose 2}= {n -2 \choose 2} +\frac{(k+1)k}{2} -\frac{(k-1)(k-2)}{2}= {n -2 \choose 2} +\frac{4k-2}{2}$ edges. Hence, any simple graph with $n$ vertices and at least $\left({{n -2 \choose 2}} +2k-1\right)+1 = \frac{(n-2)(n-3)+4k}{2}$ edges is $k$-connected. 
\end{proof} 

Note that for all $n,k\in \mathbb{N}$, $n^2-5n$ is even and $n^2-5n+6+4k>0$, hence $\frac{n^2-5n+6+4k}{2}\in \mathbb{N}$.

\end{document}